\documentclass[11pt]{amsart}
\usepackage{stmaryrd}
\usepackage{textcomp}
\usepackage{enumerate}
\usepackage{setspace}
\usepackage{amssymb}
\usepackage{amsthm}
\usepackage{amsmath}
\usepackage{graphicx}
\usepackage{extarrows}
\usepackage{marvosym}
\usepackage{empheq}
\usepackage{latexsym}
\usepackage{endnotes}
\usepackage{fontenc}
\usepackage{color}
\usepackage{comment}
\usepackage{url}

\usepackage[colorlinks=true, allcolors=blue,backref=page]{hyperref}

\begin{document}

	\newtheorem{theorem}{Theorem}
	\newtheorem{definition}[theorem]{Definition}
	\newtheorem{question}[theorem]{Question}
	
	\newtheorem{lemma}[theorem]{Lemma}
	\newtheorem{claim}[theorem]{Claim}
	\newtheorem*{claim*}{Claim}
	\newtheorem{corollary}[theorem]{Corollary}
	\newtheorem{problem}[theorem]{Problem}
	\newtheorem{fact}[theorem]{Fact}
	\newtheorem{observation}[theorem]{Observation}
	\newtheorem{proposition}[theorem]{Proposition}

    \theoremstyle{remark}
    \newtheorem{example}[theorem]{Example}
    \newtheorem{remark}[theorem]{Remark}
    \newtheorem*{remark*}{Remark}
    
	\newcommand\marginal[1]{\marginpar{\raggedright\parindent=0pt\tiny #1}}
	
	\theoremstyle{remark}

	\newcommand{\RED}{\color{red}}
	
	\newcommand{\supp}{\mathrm{supp}}
	\def\eps{\varepsilon}
	\def\HH{\mathcal{H}}
	\def\E{\mathbb{E}}
    \def\D{\mathbb{D}}
	\def\C{\mathbb{C}}
	\def\R{\mathbb{R}}
	\def\Z{\mathbb{Z}}
	\def\N{\mathbb{N}}
	\def\PP{\mathbb{P}}
	\def\l{\lambda}
	\def\s{\sigma}
	\def\t{\theta}
	\def\a{\alpha}
	\def\la{\langle}
	\def\ra{\rangle}	
	\def\Xt{\widetilde{X}}
	\def\Pt{\widetilde{P}}
	\def\Var{\mathrm{Var}}
	\def\PV{\mathrm{PV}}
	\def\NV{\mathrm{NV}}
	\def\NN{\mathcal{N}}
	\def\CC{\mathcal{C}}
	
	\def\cA{\mathcal{A}}
	\def\cQ{\mathcal{Q}}	
	\def\cC{\mathcal{C}}
	\def\F{\mathcal{F}}
	\def\tm{\tilde{\mu}}
	\def\ts{\tilde{\sigma}}
	\def\L{\Lambda}
	\def\Q{\mathcal{Q}}
	\renewcommand{\P}{\mathbb{P}}
	\renewcommand{\Re}{\operatorname{Re}}
	\newcommand{\Kt}{\widetilde{K}}
	\newcommand{\mt}{\widetilde{\mu}}
	\newcommand{\Int}{\operatorname{int}}
	\newcommand{\Ext}{\operatorname{ext}}
	\newcommand{\nut}{\widetilde{\nu}}
	\newcommand{\mes}{\mathrm{mes}}
	\newcommand{\Cov}{\operatorname{Cov}}
	\newcommand{\cE}{\mathcal{E}}

	\title[Simple proof of local universality for Kac polynomials]{A simple proof of local universality for roots of Kac polynomials}

    \author[Marcus Michelen]{Marcus Michelen}
    \author[Oren Yakir]{Oren Yakir}
    \address{Department of Mathematics, Northwestern University}
    \email{michelen@northwestern.edu}
    \address{Department of Mathematics, Massachusetts Institute of Technology}
    \email{oren.yakir@gmail.com}

	\begin{abstract}
        Let $f_n$ be a random polynomial of degree $n$ with i.i.d.\ mean-zero and finite variance random coefficients. 
        It is well known that the roots of $f_n$ cluster uniformly around the unit circle as $n$ grows large.
        We give a simple and self-contained proof of local universality for the correlation functions of the roots at the microscopic scale $1/n$ around a fixed point on the circle.
        While previous proofs of local universality were focused on studying the logarithmic potential of $f_n$, we instead directly compare the scaled random polynomial to a limiting Gaussian analytic function, and establish convergence of correlations via a soft argument, using only basic complex analysis and an anti-concentration bound of Esseen. 
	\end{abstract}
	
	\maketitle

    \section{Introduction}
    \label{sec:introduction}
    \noindent
    We consider the random Kac polynomial
    \begin{equation}
        \label{eq:def_of_kac_polynomial}
        f_n(z) = \sum_{k=0}^{n} \xi_k \, z^k \, ,
    \end{equation}
    where $\xi_0,\ldots,\xi_n$ are i.i.d.\ (possibly complex) random variables satisfying  
    \begin{equation} \label{eq:assumotion_mean_zero_variance_1}
     \E[\xi_0] = 0 \qquad \text{and} \qquad    \E[|\xi_0|^2] = 1 \, .
    \end{equation}
    Classical works~\cite{Erdos-Turan-AOM, Sparo-Sur} show that with high probability, most roots of $f_n$ tend to cluster around the unit circle as $n\to \infty$. Further, most roots are typically at distance $O(n^{-1})$ from the unit circle, and~\cite{Ibragimov-Zeitouni} computed the asymptotics for the (properly normalized) first intensity function. In short, this allows one to compute the expected number of roots in different sets at this local scale around the unit circle. 
    To understand more than just the expected number of roots, it is  natural to study the random roots at a \emph{microscopic} scale: take a point on the unit circle, rescale by a factor of $n^{-1}$ so roots become macroscopically spaced, and understand the random set of roots from the perspective of this point.  More formally, let $\zeta_0 = e^{i\theta}$ for some fixed $\theta\in(0,\pi)$\footnote{In this paper, for brevity, we only consider local universality of complex roots, though our methods are also applicable to studying real roots; see Section~\ref{subsec:real_roots} for a further discussion.}, and set
    \begin{equation}
        \label{eq:def_of_F_n}
        F_n(z) = \frac{1}{\sqrt{n}} f_n\Big( \zeta_0 + \zeta_0\frac{z}{n} \Big) \, ,
    \end{equation}
    with $f_n$ given by~\eqref{eq:def_of_kac_polynomial}. Since $F_n$ is a normalized sum of independent random variables, we expect the distribution of it (and hence, also of its roots) to be \emph{universal}, in the sense that it should not really depend on the specific law of the random coefficients. Indeed, results of this flavor were established in~\cite{Tao-Vu-IMRN, Do-Nguyen-Vu, ONguyen-Vu} under different assumptions on the random coefficients and for many other models of random functions with independent coefficients. In short, these papers show that the $k$-point correlation measures (see Section~\ref{subsec:k_point_measures} below for the definition of these measures) for the roots of $F_n$ all become asymptotically close as $n\to \infty$. This is quite useful, since when the coefficients of $f_n$ are independent Gaussian random variables, these correlations have exact integral expressions from the Kac-Rice formula, so many results proved for Gaussian random polynomials can be transferred to random polynomials with more general coefficients. 

    \subsection{Local universality result}

    The goal of this paper is to give a short and direct proof for the local universality result for the $k$-point measures, assuming only~\eqref{eq:assumotion_mean_zero_variance_1}. A key step is to identify the scaling limit for the roots of $F_n$ as $n\to \infty$. 
    The limiting point process will be given by the random zero set of the Gaussian analytic function
    \begin{equation}
        \label{eq:def_of_G}
        G(z) = \int_{0}^1 e^{zt} \, {\rm d} B_{\C}(t) \, ,
    \end{equation}
    where $B_{\C}$ is a complex-valued Brownian motion. Namely, $B_\C(t) = \tfrac{1}{\sqrt{2}}\big( B_1(t) + i B_2(t)\big)$ where $B_1,B_2$ are independent (real) standard Brownian motions. 
    For readers who are not familiar with this notation of stochastic integrals, $G:\C\to \C$ is simply a mean-zero Gaussian process with  $\E\big[G(z)G(w)\big] = 0$ for all $z,w\in \C$, and (complex) covariance kernel
    \begin{equation} \label{eq:def_of_K}
        K(z,w) = \E\big[G(z) \overline{G(w)}\big] = \int_0^1 e^{(z + \overline w)t} \,  {\rm d}t \, .
    \end{equation}
    For a multiset of complex numbers $\Lambda$ and $k\ge 1$ we denote by
    \begin{equation} \label{eq:def_of_La_k_without_diagonal}
        \Lambda_{\not=}^k = \Big\{(\lambda_1,\ldots,\lambda_k) \, : \, \lambda_{j}\in \Lambda \, , \  \lambda_{j}\not=\lambda_{j^\prime} \ \text{ for } \ j\not=j^\prime \Big\} \, .
    \end{equation}
    For an entire function $g$ we denote by $\mathcal{Z}_g$ the multiset of its zeros. 
    \begin{theorem}
        \label{thm:main_result}
        For all $k\ge 1$ and $\Phi:\C^k \to \R$ continuous with compact support we have 
        \begin{equation*}
            \lim_{n\to\infty} \E \bigg[ \sum_{ (\alpha_1,\ldots,\alpha_k) \in (\mathcal{Z}_{F_n})_{\not=}^k} \Phi(\alpha_1,\ldots,\alpha_k) \bigg] = \E \bigg[ \sum_{(\alpha_1,\ldots,\alpha_k)\in (\mathcal{Z}_G)_{\not=}^k} \Phi(\alpha_1,\ldots, \alpha_k) \bigg] \, .
        \end{equation*}
    \end{theorem}
    \noindent
    Since the expectation on the right-hand side is always finite, the following corollary is immediate. 
    \begin{corollary}
        \label{corollary:local_universality}
        Let $f_n$ and $\widetilde{f}_n$ be random polynomials of the form~\eqref{eq:def_of_kac_polynomial}, and let $F_n$ and $\widetilde F_n$ denote their corresponding local scalings~\eqref{eq:def_of_F_n}. We have that 
        \[
        \Bigg| \, \E \bigg[ \sum_{ (\alpha_1,\ldots,\alpha_k) \in (\mathcal{Z}_{F_n})_{\not=}^k} \Phi(\alpha_1,\ldots,\alpha_k) \bigg] - \E \bigg[ \sum_{ (\widetilde \alpha_1,\ldots,\widetilde \alpha_k) \in (\mathcal{Z}_{\widetilde F_n})_{\not=}^k} \Phi(\widetilde \alpha_1,\ldots,\widetilde \alpha_k) \bigg] \Bigg| \xrightarrow{n\to \infty} 0 \, ,
        \]
        for all $k\ge 1$ and for all $\Phi:\C^k \to \R$ continuous with compact support.
    \end{corollary}

    \subsection{\texorpdfstring{$k$}{k}-point measures}
    \label{subsec:k_point_measures}
    Recall that a point process $\Lambda$ in $\C$ is a random locally finite multiset of complex numbers. For every point process $\Lambda$ and $k\ge 1$, the $k$-point measure of $\Lambda$ is a measure $\alpha_{\Lambda}^k$ on $\C^k$, which is defined via the relation
    \[
    \E\bigg[ \sum_{(\lambda_1,\ldots,\lambda_k) \in \Lambda_{\not=}^k} \Phi(\lambda_1,\ldots,\lambda_k) \bigg] = \int_{\C^k} \Phi \, {\rm d} \alpha_{\Lambda}^k \, . 
    \]
    If the point process $\Lambda$ is simple (that is, if all multiplicities are equal to 1 almost surely) and the measure $\alpha_{\Lambda}^k$ is absolutely continuous with respect to Lebesgue measure on $\C^k$, then the associated density is called the $k$-point function of the process. With this terminology, Theorem~\ref{thm:main_result} is saying that for each $k\ge 1$, 
    \begin{equation} \label{eq:converge_of_k_point_measures}
    \qquad \alpha_{\mathcal{Z}_{F_n}}^k \xrightarrow{\ w^\ast \ }\alpha_{\mathcal{Z}_G}^k \qquad \qquad \text{as $n\to \infty$},
    \end{equation}
    in the vague topology of measures on $\C^k$. In fact, the limiting zero set $\mathcal{Z}_G$ is a simple point process (see Fact~\ref{fact:G_has_no_double_zeros} below), and the corresponding $k$-point functions can be computed explicitly via the celebrated Kac-Rice formulas, see~\cite[Corollary 3.4.2]{GAFbook}. 
    \subsection{Real roots and other ensembles}
    \label{subsec:real_roots}
    When the random coefficients of $f_n$ are real, the roots of $f_n$ come in conjugate pairs and $f_n$ has $\Theta(\log n)$ many real roots with high probability \cite{Kac, Erdos-Offord, Ibragimov-Maslova}. 
    Furthermore, in~\cite{Tao-Vu-IMRN, Do-Nguyen-Vu, ONguyen-Vu,michelen2021real} various local universality results for real roots of $f_n$ were derived. To keep this note brief, we focus only on complex roots, but our methods translate directly to deal with real roots as well. The key difference to point out is that when scaling near the real axis, the limiting Gaussian function is slightly different. In short, one needs to replace the complex Brownian motion $B_\C$ in~\eqref{eq:def_of_G} with a standard \emph{real} Brownian motion. Besides that, all other details of the proof are essentially the same. 

    While we do not pursue it here, it seems plausible that the general strategy we follow may be useful for proving local universality results for other models of random functions which are sums of many independent contributions.  
    
    \subsection{Previous proofs and other related works}
    As we already mentioned, different versions of Corollary~\ref{corollary:local_universality} already exist in the literature. The starting point for all existing proofs is the fact that $\tfrac{1}{2\pi}\Delta \log |F_n|$ is equal to the root count in a distributional sense. Using this observation, Tao-Vu~\cite[Theorem~5.7]{Tao-Vu-IMRN} applied a Lindeberg's swapping argument for the random coefficients to show that the distribution of $\log|F_n|$ is asymptotically universal, and with that conclude local universality of the roots. The proof in~\cite{Tao-Vu-IMRN} is quite hard and by no means elementary (e.g.\ it uses a quantitative version of Gromov's polynomial growth theorem~\cite{Shalom-Tao}). Subsequently, Do-Nguyen-Vu~\cite[Theorem~2.3]{Do-Nguyen-Vu} established a local universality result which also allows the random coefficients to have a variance profile. The argument in~\cite{Do-Nguyen-Vu} is different from~\cite{Tao-Vu-IMRN}, but relies on a highly non-trivial result of Nazarov-Nishry-Sodin~\cite{Nazarov-Nishry-Sodin} on log-integrability of Rademacher Taylor series.  
    Finally, in the recent paper~\cite{ONguyen-Vu}, Nguyen-Vu came up with a unified treatment for different local universality theorems. While the approach in~\cite{ONguyen-Vu} is more elementary compared to~\cite{Tao-Vu-IMRN, Do-Nguyen-Vu}, it still applies a sophisticated argument to prove anti-concentration bounds to deal with the logarithmic singularity in $\log|F_n|$.  
    
    On the other hand, our proof of Theorem~\ref{thm:main_result} is quite simple and direct. 
    We first identify the limiting function $G$ and prove convergence of the random zeros in law. To upgrade this convergence to the convergence~\eqref{eq:converge_of_k_point_measures}, we establish a uniform control on the tails for the local root count (thus showing uniform integrability, in an appropriate sense). The latter step is achieved by a simple argument, which involves only Jensen's formula and a classical small-ball probability bound for non-singular random vectors due to  Esseen~\cite[Theorem 6.2]{Esseen1968} (see Claim \ref{claim:esseen_small_ball_for_vectors}). 

    The appearance of the Gaussian Analytic Function $G(z)$ (given by~\eqref{eq:def_of_G}) as the (local) scaling limit of the Kac polynomials is quite intuitive. Assuming for a moment that the coefficients are complex Gaussians, we have the equality in distribution
    \begin{align} \nonumber
        F_n(z) &= \frac{1}{\sqrt{n}} \sum_{k=0}^n \xi_k \zeta_0^k  \, \Big(1+ \frac{z}{n}\Big)^k  \stackrel{d}{=} \sum_{k=0}^{n-1} \Big( B_\C\Big(\frac{k+1}{n}\Big) - B_\C\Big(\frac{k}{n}\Big) \Big)\Big(1+ \frac{z}{n}\Big)^k \, ,
    \end{align} 
    where $B_\C$ is a complex Brownian motion on $[0,1]$. Now, uniformly for $z$ in compact sets we have 
    \[
    \Big(1+\frac{z}{n}\Big)^k \approx \exp\Big(z\frac{k}{n}\Big)
    \]
    and hence
    \[
    F_n(z) \approx \sum_{k=0}^n \Big( B_\C\Big(\frac{k+1}{n}\Big) - B_\C\Big(\frac{k}{n}\Big) \Big) \exp\Big(z\frac{k}{n}\Big) \, .
    \]
    These are precisely the (random) Riemann sums which define the stochastic integral~\eqref{eq:def_of_G}, so the convergence to the GAF $G$ as $n\to \infty$ is evident. See Lemma~\ref{lemma:convergence_in_dist_for_functions} for the rigorous proof of the convergence, which in fact holds for more general random coefficients, via a simple application of the central limit theorem. 
    
    Finally, we remark that in all of the previous proofs~\cite{Tao-Vu-IMRN,Do-Nguyen-Vu, ONguyen-Vu} mentioned above, it is assumed that the random coefficients have finite $(2+\eps)$-moment for some fixed $\eps>0$, while in our result we require only finite second moment via~\eqref{eq:assumotion_mean_zero_variance_1}. It is possible to track the error term in Theorem~\ref{thm:main_result} under stronger moment assumptions on the coefficients, and in particular in~\cite{Tao-Vu-IMRN, Do-Nguyen-Vu, ONguyen-Vu} explicit bounds on the error term are given. To keep the proof as streamlined as possible, we do not attempt to optimize these bounds here. 

     \begin{remark*}
         We note that the first step of proving convergence of the random zeros in law was also established in the very recent preprint~\cite[Theorem~3.1]{Kabluchko-Khoruzhenko-Marynych} by a similar method while the present work was being completed.  The work \cite{Kabluchko-Khoruzhenko-Marynych} also allows the random coefficients to have a regularly-varying variance profile and demonstrates that the local limit changes from liquid-like to crystalline-like as the growth of the variance profile varies; \cite{Kabluchko-Khoruzhenko-Marynych} does not address convergence of $k$-point measures. 
     \end{remark*}
        
    \section{Proof of Theorem~\ref{thm:main_result}}
    \noindent
    Denote by $\cE$ the space of entire functions, and by {\sf Conf} the space of all locally finite discrete multi-sets in $\C$. Equivalently, ${\sf Conf}$ can be viewed as the set of all locally finite integer-valued measures; we use both points of view throughout. Both $\cE$ and ${\sf Conf}$ are equipped with natural Polish topologies: the topology of uniform convergence on compact sets for $\cE$, and the vague topology for ${\sf Conf}$, where elements of ${\sf Conf}$ are viewed as atomic locally finite measures.
    To each function $g\in \cE$ which does not vanish identically, we can associate $\mathcal{Z}_g\in {\sf Conf}$ the discrete multi-set of zeros in $\C$. 

    \subsection{Convergence in distribution}
    Recall that $F_n$ is given by~\eqref{eq:def_of_F_n} and that $G$ is given by~\eqref{eq:def_of_G}. Both $\{F_n\}_{n\ge 1}$ and $G$ are random elements of $\cE$, and their zero sets $\{\mathcal{Z}_{F_n}\}_{n\ge 1}$, $\mathcal{Z}_{G}$ are random elements in ${\sf Conf}$. First, we establish the corresponding convergence in the space $\cE$. 
    \begin{lemma}
        \label{lemma:convergence_in_dist_for_functions}
         We have $F_n\xrightarrow[n\to \infty]{\ d \ } G$, where convergence is in distribution as random elements of $\cE$. 
    \end{lemma}
    \noindent
    Before we turn to prove the lemma, we first show tightness of this family in $\cE$. 
    \begin{claim} \label{claim:F_n_and_derivatives_are_controlled}
        For all $R\ge 1$ and for all $n\ge 1$ we have
        $ \displaystyle
        \E \big[ \max_{|z|\le R} |F_n(z)| \, \big] \le  \exp(R) \, .
        $
    \end{claim}
    \begin{proof}
        We first Taylor expand and then apply Cauchy-Schwarz to see \begin{align*}
            \E\Big[ \max_{|z| \leq R} |F_n(z)|\Big]& = \E\Big[ \max_{|z| \leq R}  \Big| \sum_{j \geq 0} F_n^{(j)}(0) \frac{z^j}{j!}\Big| \, \Big] \\ & \leq \sum_{j \geq 0}\frac{R^j}{j!} \big(\E[|F_n^{(j)}(0)|^2]\big)^{1/2} = \sum_{j \geq 0} \frac{R^j}{j!} \Big( \frac{1}{n^{2j+1}} \sum_{k = j}^n \Big(\frac{k!}{(k-j)!}\Big)^{2} \Big)^{1/2} 
            \leq \sum_{j \geq 0} \frac{R^j}{j!} \, . \qedhere
        \end{align*}
    \end{proof}
   
    \begin{claim}
        \label{claim:correlations_converge}
        For all $z,w\in \C$ we have
        \[
        \lim_{n\to \infty} \E\big[F_n(z) \overline{F_n(w)} \, \big] = K(z,w) \, , 
        \]
        where $K$ is given by~\eqref{eq:def_of_K}. Furthermore, we have that
        $ \displaystyle
        \lim_{n\to \infty} \E\big[F_n(z) F_n(w) \big] =  0 \, .
        $
    \end{claim}
    \begin{proof}
    Set $\psi_n(z,w) = \big(1 + \frac{z}{n} \big) \big(1+\frac{w}{n} \big)  \, $ and use \eqref{eq:def_of_F_n} to compute
    \begin{align*}
        \E\big[F_n(z) \overline{F_n(w)} \, \big] &= \frac{1}{n} \sum_{k=0}^n \big(\psi_n(z,\overline{w})\big)^k = \frac{1}{n} \sum_{k=0}^n \exp\Big(\frac{k}{n} \big(z +  \overline{w}\big)\Big) + O(n^{-1}) \xrightarrow{n \to \infty} \int_{0}^1 e^{t(z+ \overline w)} \, {\rm d}t \, , 
    \end{align*}
    where the third equality is just convergence of Riemann sums to the limiting integral.  This proves the first assertion of the claim. With the above notation, we also have that
    \begin{align*}
        \left|\E\big[F_n(z) F_n(w) \, \big]\right| = \frac{|\E[\xi_0^2]|}{n} \left|\sum_{k=0}^n \zeta_0^{2k} \big(\psi_n(z,w)\big)^k \right| \leq \frac{1}{n}\left|\left(\frac{1 - \zeta_0^2}{1 - \zeta_0^2 \psi_n(z,w)}\right)\right| \xrightarrow{n \to \infty} 0 
    \end{align*}
    where $\zeta_0 = e^{i\theta}$ is the point on the unit circle around which we rescale the random polynomial $f_n$, and for the final bound we used $\theta \in (0,\pi)$. 
    \end{proof}
    \begin{proof}[Proof of Lemma~\ref{lemma:convergence_in_dist_for_functions}]
        For $\ell \ge 1$ we set
        \[
        \mathcal{K}_\ell = \bigcap_{j\ge \ell} \Big\{ f\in \cE \, : \, \max_{|z| \le j} |f(z)| \le j^2 \exp(j)  \Big\} \, . 
        \]
        By Montel's theorem~\cite[Chap.~5, \S~5.2]{Ahlfors}, $\mathcal{K}_\ell\subset \cE$ is compact for each $\ell\ge 1$. Furthermore, Claim~\ref{claim:F_n_and_derivatives_are_controlled} together with Markov's inequality implies that
        \begin{equation*}
            \P\big(F_n\not \in \mathcal{K}_\ell \big) \le \sum_{j\ge \ell} \P\Big( \max_{|z|\le j} |F_n(z)| > j^2 \exp(j) \Big) \le \sum_{j\ge \ell} \frac{1}{j^2} \, ,
        \end{equation*}
        uniformly in $n\ge 1$. This shows that the sequence of probability measures on $\cE$ induced by $\{F_n\}_{n\ge 1}$ is tight. Hence, to establish the desired convergence in distribution, it is now enough to check the convergence of the finite dimensional laws, i.e.\ to show that for all $z_1,\ldots,z_m \in \C$ fixed we have 
        \begin{equation}
            \label{eq:convergence_in_law_for_finite_dim_vectors}
            \big(F_n(z_1),\ldots, F_n(z_m)\big) \xrightarrow[n\to \infty]{ \ d \ } \big(G(z_1),\ldots,G(z_m)\big) \, ,
        \end{equation}
        where the convergence is in distribution as random vectors in $\C^m$. Claim~\ref{claim:correlations_converge} shows the corresponding correlation matrices in~\eqref{eq:convergence_in_law_for_finite_dim_vectors} converge. 
        By the Cram\'er-Wold device, the stated convergence in~\eqref{eq:convergence_in_law_for_finite_dim_vectors} will follow once we show that for all $\gamma_1,\ldots,\gamma_m \in \C$ we have 
        \begin{equation}
            \label{eq:convergence_in_law_finite_dimension_after_cramer_wold}
            \sum_{j=1}^m  \Re \big(\gamma_j \,  F_n(z_j) \big) \xrightarrow[n\to \infty]{ \ d \ } \sum_{j=1}^m  \Re \big(\gamma_j \,  G(z_j) \big)  \, .
        \end{equation}
        Indeed, note that the left-hand side of~\eqref{eq:convergence_in_law_finite_dimension_after_cramer_wold} can be written as
        \[
        \sum_{j=1}^m  \Re \big(\gamma_j \,  F_n(z_j) \big) = \frac{1}{\sqrt{n} }\Re\Big( \sum_{k=0}^n \xi_k \, v_k\Big) \qquad \text{ where } \qquad v_k = \zeta_0^k \sum_{j=1}^m \gamma_j \Big(1 + \frac{z_j}{n}\Big)^k \, .
        \]
        That is, the left-hand side of~\eqref{eq:convergence_in_law_finite_dimension_after_cramer_wold} is a large sum of independent random variables. By the Lindeberg central limit theorem~\cite[Theorem~3.4.10]{Durrett-book} for triangular arrays (noting that $\{v_k\}$ is uniformly bounded given $\{\gamma_j\}$ and $\{z_j\}$) we see the convergence~\eqref{eq:convergence_in_law_finite_dimension_after_cramer_wold}, and hence also~\eqref{eq:convergence_in_law_for_finite_dim_vectors} and we are done. 
     \end{proof}
     \noindent
     The following formulation of the continuous mapping theorem will be convenient for us.
    \begin{fact}[{\cite[Theorem 2.7]{Billingsley}}]
    \label{fact:cont_mapping_theorem}
        Suppose that $\mathcal{S}$ is a metric space and that $\{X_n\}_{n\ge 1}$ and $X$ are random elements in $\mathcal{S}$ such that $X_{n} \xrightarrow{ \ d \ } X$ as $n\to \infty$. Let $h:\mathcal{S} \to \mathcal{S}^\prime$ be a measurable map to another metric space $\mathcal{S}^\prime$, and denote by $D_h$ the set of discontinuity points of $h$. If $X\not\in D_h$ almost surely, then we also have $h(X_{n}) \xrightarrow{ \ d \ } h(X)$ as $n\to\infty$. 
    \end{fact}
    \begin{claim} \label{claim:convergence_in_law_for_zeros}
        We have $\mathcal{Z}_{F_n} \xrightarrow[n\to \infty]{ \ d \ } \mathcal{Z}_G$ in distribution as random elements of {\sf Conf}.
    \end{claim}
    \begin{proof}
        Hurwitz's theorem implies that the map $\mathcal{Z}:\cE\to {\sf Conf}$ is continuous at any $g\in \cE$ which is not identically zero. Since $G$ is almost surely a non-trivial entire function, Lemma~\ref{lemma:convergence_in_dist_for_functions} together with the continuous mapping theorem\footnote{We note that in the application of the continuous mapping theorem we use that both $\cE$ and ${\sf Conf}$ are metrizable.} (Fact~\ref{fact:cont_mapping_theorem}) implies what we want. 
    \end{proof}
    \noindent
    We say that $\Lambda\in {\sf Conf}$ is \emph{simple} if it has no multiplicities, i.e.\ if the corresponding atomic measure gives at most unit mass to singletons. 
    \begin{fact}
        \label{fact:G_has_no_double_zeros}
        Almost surely $\mathcal{Z}_G$ is simple. In other words, $G$ has only simple zeros. 
    \end{fact}
    \noindent
    Fact~\ref{fact:G_has_no_double_zeros} follows immediately from~\cite[Lemma~2.4.1]{GAFbook}, after noting that $K(z,z) > 0$ for all $z\in \C$. For the reader's convenience, we provide a simple proof below.
    \begin{proof}[Proof of Fact~\ref{fact:G_has_no_double_zeros}]
    It is enough to show that for each $R,\eps > 0$, $G$ has no double zeros in $\{|z|\le R\}$ with probability at least $1 - \eps$.  Let $M \geq 0$ be so that $\P\big(\max_{|z|\le 2R}|G''(z)| > M\big) \leq \eps/2$.  If $\displaystyle \max_{|z|\le 2R}|G''(z)| \leq M$, Taylor's theorem implies that if $G(z) = G^\prime(z) = 0$ then $$|G(w)| \leq M |w - z|^2 \, , \qquad \text{for} \quad  |w-z|\le R/2 \, .$$ In particular, for each $\delta>0$ there is a disk of radius $\delta$ about each double zero for which we have $|G(w)| \leq M \delta^2$. On the other hand, setting $S = \{w : |G(w)| \leq M \delta^2 \ \text{and} \ |w|\le 2R\}$ a simple computation shows that $\E[\text{Area}(S)] = O_{M,R}(\delta^4)$. Markov's inequality now implies that $\P(\text{Area}(S) \ge \pi \delta^2) = O(\delta^2) $, and by taking $\delta = O_{M,R}(\eps)$ we get that $G$ has no double zero in $\{|z|\le R\}$ with probability at least $1-\eps$, completing the proof.
    \end{proof}
    \noindent
    After establishing the convergence in law for the random zero sets, we now turn to upgrade this convergence to the corresponding convergence of the $k$-point measures. For that, we fix $k\ge 1$ and some $\Phi:\C^k \to \R$ a continuous test function with compact support. Define $T_{\Phi}:{\sf Conf} \to \R$ via
    \begin{equation} \label{eq:def_of_T_phi}
        T_\Phi(\Lambda) = \sum_{(\alpha_1,\ldots,\alpha_k) \in \Lambda_{\not=}^k} \Phi(\alpha_1,\ldots,\alpha_k) \, ,
    \end{equation}
    where we recall that $\Lambda_{\not=}^k$ is given by~\eqref{eq:def_of_La_k_without_diagonal}.
    \begin{lemma}
        \label{lemma:continuity_of_T_phi_on_simple}
        Suppose that $\Lambda\in {\sf Conf}$ is simple. Then for each $k\ge 1$ and $\Phi:\C^k \to \R$ continuous with compact support, $T_\Phi$ is continuous at $\Lambda$. 
    \end{lemma}
    \begin{proof}
        We need to show that for every $\{\Lambda_n\} \subset {\sf Conf}$ such that $\Lambda_n \to \Lambda$ vaguely as $n\to \infty$, we have 
        $T_{\Phi}(\Lambda_n) \xrightarrow{n\to \infty} T_{\Phi}(\Lambda) \, .$
        We denote by $\mu_n$ and $\mu$ the counting measures for $\Lambda_n$ and $\Lambda$, respectively (that is, when a unit delta mass is placed at each point of the configuration). Let $K\subset \C$ be a compact set such that $\text{supp}(\Phi) \subset K^k$, and assume without loss of generality that $\mu(\partial K) = 0$. Let $f:\C\to [0,1]$ be a continuous function such that $f \equiv 1$ on $K$ and has compact support. Then vague convergence of $\Lambda_n$ to $\Lambda$ implies that
        \[
        \mu_n(K) \le \int_{\C} f \, {\rm d}\mu_n \xrightarrow{n\to \infty} \int_{\C} f \, {\rm d}\mu < \infty \, .
        \]
        Therefore, the sequence of finite measures $\mu_n \large|_K$ is tight, and in fact $\mu_n\large|_K\xrightarrow{n\to\infty} \mu\large|_K$ in the sense of weak convergence\footnote{Recall this is a slightly stronger notion, as we are allowed to test against \emph{bounded} continuous functions rather than just \emph{compactly support} continuous functions.}. In particular, the Portmanteau theorem now implies that
        $ \displaystyle       \lim_{n\to \infty}\mu_n(B) = \mu(B) \, ,$ 
        for any Borel set $B\subset K$ such that $\mu(\partial B) = 0$. This means that for all $n$ large enough we must have $\mu_n(K) = \mu(K)$.  Write $\Lambda|_K = \{\lambda_1,\ldots,\lambda_N\}$ with $N = \mu(K)$ and note that for each fixed $\eps > 0$ sufficiently small, we may take  $B = B_\eps(\lambda_j)$ and see $\mu_n(B_\eps) \to \mu(B)$.  In particular, this means for each $\eps > 0$ sufficiently small, for all $n$ large enough we may list $\Lambda_n|_K = \{\lambda_1^{(n)},\ldots,\lambda_N^{(n)}\}$ so that $|\lambda_j^{(n)} - \lambda_j| \leq \eps$.  By continuity of $\Phi$ this implies $T_\Phi(\Lambda_n) \to T_\Phi(\Lambda)$ as $n\to\infty$.
    \end{proof}
    \begin{claim}
        \label{claim:convergence_in_law_of_T_phi}
        Let $T_\Phi$ be given by~\eqref{eq:def_of_T_phi}. We have
        $
        T_{\Phi}(\mathcal{Z}_{F_n}) \xrightarrow[n\to \infty]{ \ d \ } T_\Phi(\mathcal{Z}_G) \, .
        $
    \end{claim}
    \begin{proof}
        Fact~\ref{fact:G_has_no_double_zeros} implies that $\mathcal{Z}_G$ is almost surely simple, and hence Lemma~\ref{lemma:continuity_of_T_phi_on_simple} implies that almost surely $\mathcal{Z}_G$ is a continuity point of $T_{\Phi}$. The claim now follows by applying the continuous mapping theorem with $h = T_{\Phi}$, which gives the desired convergence via Claim~\ref{claim:convergence_in_law_for_zeros}.
    \end{proof}
    \subsection{Uniform integrability}
    To prove Theorem~\ref{thm:main_result}, we need to upgrade the stated convergence in Claim~\ref{claim:convergence_in_law_of_T_phi} to a statement about convergence of expectations. For that, we establish uniform integrability via Lemma~\ref{lemma:uniform_bound_on_k_moments} below. 
    For $R\ge 1$ we denote by $\D_R = \{ |z| \le R \}$.
     \begin{lemma}
        \label{lemma:uniform_bound_on_k_moments}
        For all $R\ge 1$ and $k\ge 1$ we have
        $ \displaystyle
        \sup_{n\ge 1} \E\big[ \mathcal{Z}_{F_n}\big( \D_R \big)^k \big] < \infty \, .
        $
    \end{lemma}
    \noindent
    To prove Lemma~\ref{lemma:uniform_bound_on_k_moments}, we will need a classical small-ball bound for sum of non-degenerate random vectors, which follows essentially from the work of Esseen~\cite{Esseen1968}.  We reprove it here for completeness.
    \begin{claim} \label{claim:esseen_small_ball_for_vectors}
    Let $V_0,V_1,\ldots,V_n \in \R^d$ be vectors with $\|V_j\| \leq M$ for all $j$. Define the matrix $\Sigma_n = \frac{1}{n} \sum_{j = 0 }^n V_j V_j^T \in \R^{d \times d}$ and suppose $\langle \eta, \Sigma_n \eta\rangle \geq \kappa \|\eta\|^2$ for all $\eta \in \R^d$.  
    Suppose that $\xi_j$ are i.i.d.\ non-degenerate random variables in $\C$, all with common distribution $\xi$.  Then
    $$\P\Big(\big\|\sum_{j=0}^n \xi_j V_j \big\| \leq 1\Big) \le C \, n^{-d/2}  \, , $$
    where $C>0$ depends only on $d, \kappa, M$ and $\xi$.
    \end{claim}
    \begin{proof}
    Since $\xi$ is non-degenerate, either the real or imaginary part of $\xi$ is non-degenerate so assume without loss of generality that the real part of $\xi$ is non-degenerate; then it is sufficient to assume $\xi_j \in \R$. 
    Let $\varphi_\xi(t) = \E[\exp(it\xi)]$ denote the characteristic function of $\xi$.  Esseen's inequality~\cite{Esseen1968} (see also~\cite[Lemma~7.17]{Tao-Vu-Book} for a textbook treatment) shows that for each $s \geq 0$
    \begin{equation} \label{eq:small_ball_bound_after_esseen}
        \max_{w \in \R^d}\P\Big(\big\|\sum_{j=0}^n \xi_j V_j - w\big\| \leq s\Big) \le C_d \int_{\{\|\eta\| \le \frac{s}{4}\}} \Big|\prod_{j=0}^n \varphi_\xi\big(\langle V_j , \eta \rangle \big)\Big| \, {\rm d}m (\eta) \, ,
    \end{equation}
    where $m$ is the Lebesgue measure on $\R^d$.  By covering the unit ball with $O_{d,s}(1)$ many balls of radius $s > 0$, it is sufficient to bound the right-hand-side of \eqref{eq:small_ball_bound_after_esseen} by $O(n^{-d/2})$ for some fixed $s = s(M,d,\xi) > 0$. Let $\xi^\prime$ be an independent copy of the random variable $\xi$. Since $\xi$ has a non-degenerate distribution, there exists $a>0$ so that $\P(a^{-1} \le |\xi -\xi^\prime| \le a) >0$, and we get that 
    \begin{align*}
        \big|\varphi_\xi(t)\big|^2 = \big|\E\cos((\xi-\xi^\prime) t)\big| \le \exp(-c_0t^2)
    \end{align*}
    for some $c_0>0$ depending only on the distribution of $\xi$ and for all $|t|\le \tfrac{1}{4aM}$.  Setting $s = 1/(aM)$ we may bound 
    \begin{align*}
       \int_{\{\|\eta\| \le \frac{s}{4}\}} \Big|\prod_{j=0}^n \varphi_\xi\big(\langle V_j , \eta \rangle \big)\Big| \, {\rm d}m (\eta) &\le \int_{\{\|\eta\| \le \frac{s}{4}\}}  e^{-c\sum_{j=0}^n \langle V_j , \eta\rangle^2} \, {\rm d}m (\eta) =\int_{\{\|\eta\| \le \frac{s}{4}\}}  e^{-c n \, 
        \langle \eta , \Sigma_n\eta\rangle} \, {\rm d}m (\eta) \\ &\le \int_{\{\|\eta\| \le \frac{s}{4}\}}  e^{-c \kappa n \, 
       \|\eta\|^2} \, {\rm d}m (\eta) \le C' n^{-d/2} \, . \qedhere
    \end{align*}
    \end{proof}
    \noindent
    Next, we show how to get Lemma~\ref{lemma:uniform_bound_on_k_moments} from Esseen's bound.
    \begin{proof}[Proof of Lemma~\ref{lemma:uniform_bound_on_k_moments}]
    By possibly increasing $R$ by a factor which depends only on $k\ge 1$, we may assume that there exists $0\le \varphi_1<\varphi_2<\cdots<\varphi_{4k}\le R$ so that
    \[
    \varphi_{j+1} - \varphi_j \ge 8k \, , \qquad \text{for all } \ j=1,\ldots, 4k.
    \]
    Since the map $R\mapsto \mathcal{Z}_{F_n}(\D_R)$ is non-decreasing, proving the conclusion of the lemma with this (possibly larger) $R=R_k$ suffices.  We use Claim \ref{claim:esseen_small_ball_for_vectors} to show that $F_n$ is typically not too small on at least on point of the set $\{i\varphi_j\}_{j=1}^{4k}$. Define
    \[ 
    V_j = \Re\left(\zeta_0^j\cdot \Big(\Big(1+\frac{i\varphi_1}{n}\Big)^j , \ldots, \Big(1+\frac{i\varphi_{4k}}{n}\Big)^j \Big)\right) \in \R^{4k} \, .
    \]
    Note first that $\|V_j\| = O_k(1)$.  To verify the non-singularity condition of Claim~\ref{claim:esseen_small_ball_for_vectors}, we see that $\Sigma_n$ is the covariance matrix for the vector
    $\Big(\Re F_n(i\varphi_1),\ldots,\Re F_n(i\varphi_{4k})\Big) \,.$
    Claim~\ref{claim:correlations_converge} implies that it suffices to show uniform non-singularity when $F_n$ is replaced by $G$. Let us write this matrix $\Sigma(j,j') = \E[\Re G(i \varphi_j) \Re G(i \varphi_{j'})]$.  Note that $\Sigma(j,j) = 1/2$ and for $j \neq j'$ we have 
    \[
    |\Sigma(j,j')| \leq |K(i\varphi_j,i\varphi_{j^{\prime}})| = \Big|\int_{0}^1 e^{it(\varphi_j - \varphi_{j^\prime})} {\rm d}t \Big| \le \frac{2}{|\varphi_j - \varphi_{j^\prime}|} \le \frac{1}{4k} \, .
    \]
    For every $\eta \in \R^d$ such that $\|\eta\| = 1$ we then have $$1/2 - \langle \eta, \Sigma \eta \rangle = \sum_{j' \neq j} v_j \Sigma_{j,j'} v_{j'} \leq \frac{1}{4k}\sum_{j' \neq j} |v_j||v_{j'}| \leq \frac{1}{4}$$
    where the last bound is by Cauchy-Schwarz. By scaling, this shows $\langle \eta , \Sigma \eta \rangle \geq (1/4) \|\eta\|^2$ for all $\eta \in \R^d$, which gives the desired non-singularity.  If we set $\mathcal{B} = \{|F_n(i \varphi_j)| \leq  n^{-1/2} \, , \,  \text{ for all } j = 1,\ldots,4k\}$, then Claim \ref{claim:esseen_small_ball_for_vectors} shows $\P(\mathcal{B}) \leq C_k n^{-2k}$. Furthermore, the classical Littlewood-Offord lemma  (which follows from the case $d=1$ in Claim~\ref{claim:esseen_small_ball_for_vectors} via a simple covering argument, see~\cite[Corollary 1]{Esseen1968}) combined with a union bound gives
    \[
    \P \Big( |F_n(i\varphi_j)| \le \eps \ \text{ for some } j =1,\ldots,k\Big)\le C_k \, \eps
    \]
    for all $\eps\ge n^{-1/2}$. On the other hand, the classical Jensen bound~\cite[Chap.~5]{Ahlfors} shows that
    \begin{equation*}
    \mathcal{Z}_{F_n}(\D_R) \le C \log \frac{\max_{|z|\le 2R }|F_n(z)|}{\max_{|z|\le R} |F_n(z)|} \, .
    \end{equation*}
    Claim~\ref{claim:F_n_and_derivatives_are_controlled} implies that $\E\big[\displaystyle\max_{|z|\le 2R}\log^k|F_n(z)|)\big] \le C(k,R)$, and with the above we get that
    \begin{align*}
        &\E\Big[ \mathcal{Z}_{F_n}\big( \D_R \big)^k \Big] \\ &\le \E\Big[ \mathcal{Z}_{F_n}\big( \D_R \big)^k \mathbf{1}_{\{\exists j\in[k] \, : \, |F_n(i\varphi_j)| \ge 1 \}} \Big]   + \sum_{\ell=1}^{\lfloor \tfrac{1}{2} \log_2(n) \rfloor } \E\Big[ \mathcal{Z}_{F_n}\big( \D_R \big)^k \mathbf{1}_{\{\exists j\in[k] \, : \, |F_n(i\varphi_j)| \in [2^{-\ell},2^{-\ell + 1}) \}} \Big] + n^k \, \P(\mathcal{B}) \\ & \lesssim_{k,R} \sum_{\ell = 1 }^{\lfloor \tfrac{1}{2} \log_2(n) \rfloor } \ell^{k} \,  \P\big(\exists j\in[k] \, : \, |F_n(i\varphi_j)| \le 2^{-\ell + 1} \big) \lesssim_{k,R} \sum_{\ell\ge 1} \ell^{k} 2^{-\ell} \le C(k,R) \, . \qedhere
    \end{align*}
    \end{proof} 
    \allowdisplaybreaks
    \begin{proof}[Proof of Theorem~\ref{thm:main_result}]
        Without loss of generality we may assume that $\Phi$ is non-negative and has $\Phi \le 1$. For $M\ge 1$ we set $g_M(x) = \min\{x,M\}$, and note that Claim~\ref{claim:convergence_in_law_of_T_phi} implies that
        \begin{equation} \label{eq:convergence_of_expectations_with_truncations}
            \lim_{n\to \infty} \E\big[g_M\big(T_{\Phi}(\mathcal{Z}_{F_n})\big)\big] = \E\big[g_M\big(T_{\Phi}(\mathcal{Z}_{G})\big)\big] \, .
        \end{equation}
        Note also that by the monotone convergence theorem we have $\E\big[g_M\big(T_{\Phi}(\mathcal{Z}_{G})\big)\big] \xrightarrow{M \to \infty} \E\big[T_{\Phi}(\mathcal{Z}_{G})\big]\,.$  Thus, to show $\E\big[T_{\Phi}(\mathcal{Z}_{F_n})\big] \xrightarrow{n \to \infty}  \E\big[T_{\Phi}(\mathcal{Z}_{G})\big]$, it is enough to prove \begin{equation}\label{eq:F_n-truncation-limit}
            \lim_{M\to \infty} \limsup_{n\to \infty} \Big|\E\big[T_{\Phi}(\mathcal{Z}_{F_n})\big] - \E\big[ g_M\big(T_{\Phi}(\mathcal{Z}_{F_n})\big) \big]\Big| = 0 \,.
        \end{equation}
        To prove \eqref{eq:F_n-truncation-limit}, take $R\ge 1$ large enough so that $\text{supp}(\Phi)\subset \D_R^k$, and note $|T_\Phi(\mathcal{Z}_n(\D_R))| \leq \mathcal{Z}_n(\D_R)^k$ since $\Phi \leq 1$.  By Cauchy-Schwarz we may bound \begin{align*}
            &\Big|\E\big[T_{\Phi}(\mathcal{Z}_{F_n}) - g_M\big(T_{\Phi}(\mathcal{Z}_{F_n})\big) \big]\Big| \\ &\leq \E\Big[\mathbf{1}_{\{\mathcal{Z}_{F_n}(\D_R)^k \ge M\}} \mathcal{Z}_{F_n}(\D_R)^k\Big] \leq \left(\P\big(\mathcal{Z}_{F_n}(\D_R)^k \geq M\big)\cdot \E[\mathcal{Z}_{F_n}(\D_R)^{2k}] \right)^{1/2} \leq \frac{\sup_{n} \E[\mathcal{Z}_{F_n}(\D_R)^{2k}]}{M} \, , 
        \end{align*}
        where in the last step we used Markov's inequality. In view of Lemma~\ref{lemma:uniform_bound_on_k_moments}, the above display confirms \eqref{eq:F_n-truncation-limit} and completes the proof.
    \end{proof}

        \subsection*{Acknowledgments} MM is supported in part by NSF grants DMS-2336788 and DMS-2246624. OY is supported in part by NSF postdoctoral fellowship DMS-2401136.

\bibliographystyle{abbrv}
\bibliography{random_polynomials}
	
\end{document}